\documentclass{article}
\usepackage{amsmath}
\usepackage{amssymb}
\usepackage{amsthm}
\usepackage{amscd}
\usepackage{amsfonts}
\usepackage{graphicx}
\usepackage{fancyhdr}
\usepackage{epsfig}
\usepackage{hyperref}
\usepackage{makeidx}
\usepackage{setspace}
\usepackage{subfigure}
\usepackage{empheq}
\usepackage{verbatim}
\usepackage{tikz}
\usepackage{xcolor}
\usepackage{lineno}

\theoremstyle{plain} \numberwithin{equation}{section}
\newtheorem{theorem}{Theorem}[section]
\newtheorem{corollary}[theorem]{Corollary}
\newtheorem{conjecture}[theorem]{Conjecture}
\newtheorem{lemma}[theorem]{Lemma}
\newtheorem{proposition}[theorem]{Proposition}

\theoremstyle{definition}
\newtheorem{definition}[theorem]{Definition}

\newtheorem{remark}[theorem]{Remark}
\newtheorem{example}[theorem]{Example}

\newcommand{\link}{\textup{link}}
\newcommand{\del}{\textup{del}}
\newcommand{\ds}{\displaystyle}
\newcommand{\redHo}{\widetilde{\textup{H}}}
\newcommand{\ld}{\lessdot}
\newcommand{\smsm}{\smallsetminus}

\newcommand{\ra}{\rightarrow}
\newcommand{\lra}{\longrightarrow}
\newcommand{\mbf}{\mathbb}
 
\newcommand{\Ho}{\widetilde{H}}

\newcommand{\rank}{{\mathrm{rank}}}
\newcommand{\LN}{\mathcal{L}(n)}
\newcommand{\cl}[1]{\lceil {#1}\rceil}

\DeclareMathOperator{\maxRanked}{MR  }
\DeclareMathOperator{\LCM}{LCM}

\DeclareMathOperator{\image}{im }

\DeclareMathOperator{\Star}{star }

\title{The Betti poset in monomial resolutions}
\author{Timothy Clark and Sonja Mapes}

\begin{document}
\maketitle

\begin{abstract}
Let $P$ be a finite partially ordered set with unique minimal 
element $\hat{0}$. We study the Betti poset of $P$, created 
by deleting elements $q\in P$ for which the open interval 
$(\hat{0}, q)$ is acyclic. Using basic simplicial topology, we demonstrate an isomorphism 
in homology between open intervals of the form $(\hat{0},p)\subset P$ 
and corresponding open intervals in the Betti poset.
Our motivating application is that the Betti poset of a monomial ideal's lcm-lattice 
encodes both its $\mathbb{Z}^{d}$-graded 
Betti numbers and the structure of its minimal free resolution. 
In the case of rigid monomial ideals, we use the data of the Betti poset to 
explicitly construct the minimal free resolution. Subsequently, we introduce 
the notion of rigid deformation, a generalization of 
Bayer, Peeva, and Sturmfels' generic deformation. 
\end{abstract}

\section*{Introduction} 

There has been a great deal of work on the problem, posed by Kaplansky
in the early 1960s, of constructing a minimal free resolution of a monomial
ideal $M$ in a polynomial ring $R$.  While there are several methods for
computing the Betti numbers, 
a method for constructing the maps in a
minimal resolution remains elusive except for specific classes of
ideals.  

Our contributions to this problem are as follows.  In
Theorem \ref{BettiSupportsRigid}, we give a construction for the minimal free resolution
of a rigid monomial ideal.  We follow this by defining \emph{rigid deformations}, which leverage 
this construction to give minimal free resolutions of non-rigid monomial ideals. 
This program is easily shown successful in the case of simplicial resolutions, 
which we verify in Theorem \ref{simplicialRigidDefo}.  Combined with knowledge 
of join-preserving maps between finite atomic lattices, this 
gives a structural understanding of the pathways and obstructions to minimally resolving 
monomial ideals by discerning the resolution structure of combinatorially similar ideals. 

Our overal program mirrors that of generic monomial ideals
\cite{BPS,MSY}, which are known to have their
minimal resolutions supported on the Scarf simplicial complex $\Delta_M$. 
In this case, the maps in the minimal resolution are formed by 
$\mathbb{Z}^d$-grading the maps in the algebraic chain complex of $\Delta_M$. 
Bayer, Peeva, and Sturmfels \cite{BPS} also introduce the notion of generic deformation of exponents, 
which gives a (generally non-minimal) simplicial resolution for every ideal, 
supported on the Scarf complex of the target generic ideal. 
The benefit of our approach is an ability to describe (and in turn deform to) 
resolutions whose structure is more complicated than the simplicial resolutions of generic ideals. 

The second author establishes the motivation for our approach by 
showing in \cite{mapes} that generic deformation has blind spots. First,  
a poor choice of deformed exponents can result in a generic ideal whose minimal resolution 
is as large as possible -- the Taylor complex \cite{Taylor}. Moreover, there exist finite 
atomic lattices whose ideals have simplicial resolutions, but are not the result of any 
generic deformation. In particular, there exist even monomial ideals of projective dimension two 
whose minimal resolutions have a very basic simplicial structure, but which cannot be attained by the 
process of generic deformation. Our approach overcomes these deficiencies by taking advantage 
of the lattice structure of the set of finite atomic lattices on a fixed number of atoms due to Phan\cite{phan}. 
In particular, we investigate the role that rigid ideals take in this lattice. 

The main tool in our description of the minimal resolution of rigid ideals is the order-theoretic perspective 
taken by the first author in \cite{Clark}. The technique presented therein creates a complex of vector 
spaces from the homology of intervals in a poset 
and yields minimal resolutions in some cases. It recovers the resolutions given by several 
known construction methods as being supported on a relevant poset, including the Scarf resolution for generic ideals. 

To describe the resolution structure of a rigid ideal, we identify the combinatorial object which encodes the unique 
$\mathbb{Z}^d$-graded bases in a rigid ideal's minimal resolution, extending our initial results \cite{CMRigid} on 
rigid ideals. Where the Scarf complex is the unique object supporting the minimal resolution of 
a generic ideal, the \emph{Betti poset} plays the same role for a rigid ideal. The Betti poset is the 
subposet of an ideal's lcm-lattice which does not contain the homologically irrelevant data present in the lcm-lattice. 
We provided initial examples of the utility of the Betti poset in \cite{CMRigid}, and this object has been subsequently 
studied in \cite{TchernevVarisco} using techniques of category theory. 

Fundamental progress comes in Section \ref{S:BettiPoset} by proving general statements about the Betti poset of a poset 
$P$. Specifically, Theorem \ref{removeNC} uses classical techniques from 
simplicial homology to establish an isomorphism between the homology of certain open intervals 
of a poset $P$ and those analogous intervals appearing in its Betti poset.  

In Section \ref{S:BettiOfLCM}, we turn our attention to the lcm-lattice of a monomial ideal $M$ 
and its Betti poset $B_M$. First, we establish Theorem \ref{BettiMFR} by using basic relabeling 
techniques to show that the 
isomorphism class of $B_M$ completely determines its minimal free resolution. 
Our proof avoids the functorial techniques of \cite{TchernevVarisco}.  

Our main result is Theorem \ref{BettiSupportsRigid}, which states that when an 
ideal is rigid, the Betti poset of its lcm-lattice supports the minimal free resolution. 
Section \ref{S:MinResRigid} contains the technical proof of Theorem \ref{BettiSupportsRigid}. 

With results on rigid ideals in hand, we introduce the notion of rigid deformation by appealing to the 
structure of $\LN$, the lattice of finite atomic lattices introduced by Phan \cite{phan}. Indeed, 
as a corollary to Theorem \ref{BettiSupportsRigid}, we see that for every monomial ideal $M$ 
whose lcm-lattice is comparable in $\LN$ to that of a rigid ideal, the Betti poset of the rigid 
ideal supports the minimal resolution of any monomial ideal whose Betti poset is isomorphic to $B_M$. 
This generalizes the equivalent Theorem \ref{BettiMFR} and the results of \cite{TchernevVarisco}. 
For a non-rigid ideal, there may be several rigid ideals whose Betti poset supports a minimal resolution. 
Conversely, we present an example of a non-rigid ideal which admits no rigid deformation. 

Throughout the paper, all posets are finite with a unique minimal element $\hat{0}$. In order to focus attention 
on the homological properties of their order complexes, when considering the topological and homological 
properties of the order complex $\Delta(P)$ we write $P$ in its place when the context is clear. 

\section{Homological properties of the Betti poset}\label{S:BettiPoset}

In the category of simplicial complexes, the notion of vertex deletion is well-studied. We review 
the analogous concept for deleting elements from a poset $P$. Using this process, we show 
that under certain assumptions, deleting particular elements from a fixed open interval in a 
poset induces an isomorphism in homology. 

Let $p$ be an element of the poset $P$. We refer to the poset  
$P\smsm\{p\}$ as the \emph{deletion of $p$ from $P$} since 
in the order complex of (intervals of) $P$, this operation 
corresponds to the topological notion of vertex deletion. 
The idea of this procedure is related to the one described by 
Fl\"oystad \cite{Floystad1}. 

In what follows, we examine the reduced homology of order complexes of posets. For a poset $P$, we write 
$h_i=h_i(P)=\dim_\Bbbk\Ho_i(\Delta (P),\Bbbk)$ for the $\Bbbk$-vector space dimension of the homology of the 
order complex of the poset $P\smsm\{\hat{0}\}$.  
When $p \in P$, for notational simplicity we write $\Delta_p=\Delta_p^P=\Delta(\hat 0,p)$ 
for the order complex of the associated open interval in $P$. 
As is standard, we say that an element $x\in{P}$ is \emph{covered} by $y$ 
(which we write as $x\ld y$) when $x<y$ and there exists no $z\in{P}$ such that $x<z<y$. 
For those $p\in P$ with the property that $h_i({\Delta_p})=0$ for every $i$, we 
say that $p$ is a \emph{non-contributor} to the homological data of $P$. 

We now apply the notions of simplicial deletion, link, and star to the order complex of a poset.
Focusing on the order complex $\Delta_q$ and a poset element $p\in(\hat{0},q)$ recall 
$$\del_{\Delta_q}(p)=\{\sigma\in\Delta_q:p\notin\sigma\},$$  
$$\link_{\Delta_q}(p)=\{\sigma\in\Delta_q:\sigma\cup\{p\}\in\Delta_q\},\textup{ and}$$  
$$\Star_{\Delta_q}(p)=\{\sigma\in\Delta_q:p\in\sigma\}.$$  

Our interest in $\link_{\Delta_q}(p)$ comes from its relationship to
$\del_{\Delta_q}(p)$ and $\Star_{\Delta_q}(p)$, which will be of use in analyzing the 
homological data in $P$. 

When restricting to the subposet $(\hat{0},q)\subset P$, the complex $\link_{\Delta_q}(p)$ 
is the simplicial join $\Delta(0,p)*\Delta(p,q)$. Using this notion, we prove the following. 

\begin{lemma}\label{removeOne}
Let $P$ be a poset and fix $q \in P$.
If $p<q\in P$ has the property that $h_j({\Delta_p})=0$ for every $j$, 
then there is an isomorphism in homology 
$\Ho_i(\Delta_q,\Bbbk)\cong\Ho_i(\del_{\Delta_q}(p),\Bbbk)$.
\end{lemma}  

\begin{proof}
Take $q\in P$ and recall that for any element $p<q$ in $(\hat{0},q)$, we have 
$\link_{\Delta_q}(p)=\Delta_p*\Delta(p,q)$. 
Using the Kunneth formula, the reduced homology of $\link_{\Delta_q}(p)$ is 
$$\Ho_k(\link_{\Delta_q}(p))=\bigoplus_{i+j=k-1}\Ho_{i}(\Delta_p)\otimes\Ho_{j}(\Delta(p,q)).$$ 
Since $p$ is assumed to be a non-contributor, $\Ho_{i}(\Delta_p)=0$ for every $i$ and 
therefore $\Ho_k(\link_{\Delta_q}(p))=0$ for every $k$. 

Next, consider the Mayer-Vietoris sequence in reduced homology for 
the triple $$\left(\del_{\Delta_q}(p),\Star_{\Delta_q}(p),\Delta_q\right). $$
Since $\link_{\Delta_q}(p)=\del_{\Delta_q}(p)\cap\Star_{\Delta_q}(p)$ 
has homology which vanishes 
in every dimension, the Mayer-Vietoris sequence 
$$
\begin{aligned}
\cdots \rightarrow 
\Ho_i(\link_{\Delta_q}(p))\,\rightarrow\, &
\Ho_i(\del_{\Delta_q}(p)) \oplus \Ho_i(\Star_{\Delta_q}(p))\, \hspace{2cm}\\ 
& \\
& \rightarrow\Ho_i(\Delta_q) \rightarrow \Ho_{i-1}(\link_{\Delta_q}(p))\,\rightarrow\cdots 
\end{aligned}
$$ 
reduces for each $\ell$ to 
\[ 0  
\rightarrow \Ho_\ell (\del_{\Delta_q}(p)) \oplus \Ho_\ell (\Star_{\Delta_q}(p)) 
\stackrel{\psi_\ell}{\rightarrow} \Ho_\ell (\Delta_q) 
\rightarrow 0 \]
Hence, the map $\psi_\ell$ is an isomorphism. Furthermore, the simplicial complex $\Star_{\Delta_q}(p)$ 
is a cone with apex $p$ and therefore $\Ho_\ell (\Star_{\Delta_q}(p)) = 0$ for all $\ell$.  
Thus, $\Ho_\ell(\del_{\Delta_q}(p),\Bbbk)\cong\Ho_\ell(\Delta_q,\Bbbk)$ for all $\ell$. 
\end{proof}  

When $\Delta_p$ has trivial homology, the 
isomorphism between the homology of the order complex of the original open interval 
$(\hat{0},q)$ and the homology of the order complex of the poset $(\hat{0},q)\smsm\{p\}$ 
suggests the following definition

\begin{definition}
The \emph{Betti poset} of a poset $P$ is the subposet 
consisting of all homologically contributing elements, 
$$B(P)=\{q\in P \ \vert\ \Ho_i(\Delta_q) \neq 0 \textup{ for at least one } i\}.$$
\end{definition}

The name \emph{Betti poset} is motivated by the study of minimal resolutions 
of monomial ideals. This object was introduced in \cite{CMRigid} and studied in 
\cite{TchernevVarisco}. We specialize to monomial ideal setting in Section \ref{S:BettiOfLCM}. 
In particular, if $P$ is a finite atomic lattice, then it is poset-isomorphic to the 
lcm-lattice of a monomial ideal. In this context, the elements of the Betti poset of $P$ 
consist of those multidegrees of a monomial ideal which contribute Betti numbers in the 
minimal free resolution. Before turning to this application, we prove several general facts about Betti posets. 

\begin{corollary}\label{SameBettis}
Suppose $p\in P$ is a noncontributor, i.e. that $h_i (\Delta_p) = 0$ for all $i$. 
For $P'=P\smsm\{p\}$ with the induced partial ordering, $B(P)=B(P')$. 
\end{corollary}

\begin{proof}
Suppose $p\in P$ has the property that $h_i(\Delta_p^P) = 0$ for every $i$. 
For those $q \in P$ such that $q < p$, or when $q$ is not comparable to $p$, 
the interval $(\hat{0}, q)$ is the same in both $P$ and $P'$. As such, 
the isomorphism on order complexes induces an isomorphism of homology. 
Thus, if $q < p$ then $q\in B(P)$ if and only if $q\in B(P')$. 

For $q > p$, then according to Lemma \ref{removeOne}, 
$\Ho_i(\Delta_q^P)\cong\Ho_i(\del_{\Delta_q^P}(p))$ for every $i$. 
However, $\del_{\Delta_q^P}(p) = \Delta_q^{P'}$ and therefore  
$\Ho_i(\Delta_q^P) \cong \Ho_i\left(\Delta_q^{P'}\right)$. 
Hence, if $q > p$ then $q\in B(P)$ if and only if $q\in B(P')$. 
\end{proof}

\begin{theorem}\label{removeNC}
Let $P$ be a poset and $B(P)$ its Betti poset. 
For each $q\in B(P)$ we have an isomorphism of $\ \Bbbk$-vector spaces 
$\Ho_i(\Delta_q^P) \cong \Ho_i\left(\Delta_q^{B(P)}\right)$.
\end{theorem}

\begin{proof}
By induction on the number of non-contributing elements in $P$. 

The base case where $P$ and $B(P)$ differ by only 
one non-contributing element $p$, is a special case of Corollary \ref{SameBettis}. 
Thus, for every $q \in B(P)$, we have 
$\Ho_i(\Delta_q^P) \cong \Ho_i\left(\Delta_q^{B(P)}\right)$. 

Let $k>1$ and suppose that for any poset $P'$ which 
has fewer than $k$ non-contributing elements, 
$\Ho_i(\Delta_q^{P'}) \cong \Ho_i\left(\Delta_q^{B(P')}\right)$ for every $q\in B(P')$. 
Let $P$ be a poset which has $k$ non-contributing elements. Write $p$ for a 
non-contributor of $P$, so that we have $P=P'\cup\{p\}$ for some $p$. 
Taking $q\in B(P)$ we have two possibilities, similar to the proof of Corollary \ref{SameBettis}. 

First, if $q<p$ or if $q$ is not comparable to $p$, then $(\hat{0}, q)$ is the 
same in both $P'$ and $P$, inducing an isomorphism in homology 
$\Ho_i(\Delta_q^{P})\cong\Ho_i(\Delta_q^{P'})$. Using the induction hypothesis 
in concert with the fact that Corollary \ref{SameBettis} guarantees $B(P)=B(P')$, 
we have 
$$\Ho_i(\Delta_q^{P})\cong\Ho_i(\Delta_q^{P'})
\cong \Ho_i\left(\Delta_q^{B(P')}\right)\cong \Ho_i\left(\Delta_q^{B(P)}\right).$$ 

Alternately, if $q>p$ then 
Lemma \ref{removeOne} applies, and viewing each element in the larger poset 
$P$, we have $\Ho_i(\Delta_q^{P})\cong\Ho_i(\del_{\Delta_q^{P}}(p))$. 
Since $\del_{\Delta_q^P}(p) = \Delta_q^{P'}$, then we obviously have 
$\Ho_i(\del_{\Delta_q^{P}}(p))\cong\Ho_i(\Delta_q^{P'})$. 
The induction hypothesis now applies to $P'$, a poset 
with fewer than $k$ non-contributors. Together with the equality of posets 
$B(P)=B(P')$ guaranteed by Corollary \ref{SameBettis},  we have 
$$\Ho_i(\Delta_q^{P})\cong\Ho_i(\del_{\Delta_q^{P}}(p))\cong\Ho_i(\Delta_q^{P'})
\cong \Ho_i\left(\Delta_q^{B(P')}\right)\cong \Ho_i\left(\Delta_q^{B(P)}\right),$$
which completes the proof. 
\end{proof}

\section{Resolutions of monomial ideals}\label{S:BettiOfLCM}

We now use the results of Section \ref{S:BettiPoset} to study the minimal free resolution of 
a monomial ideal $M$ in a polynomial ring $R$. 
Recall that the lcm-lattice of $M$ is the set $L_M$ of least common multiples 
of the $n$ minimal generators of $M$, with minimal element $1\in R$ and ordering given by divisibility. 
For the remainder of the paper, we consider the Betti poset of an lcm-lattice, which we denote 
$B_M=B(L_M)$. 

With the appropriate notions established, we state the first commutative algebra result of this paper. 

\begin{theorem}\label{BettiMFR}
Suppose $M\subset R=\Bbbk[x_1,\ldots,x_d]$ and $N\subset S=\Bbbk[y_1,\ldots,y_{t}]$ 
are monomial ideals such that $B_M\cong B_N$. 
A minimal resolution of $M$ can be relabeled to give the minimal resolution of $N$. 
\end{theorem} 

\begin{proof} 
Suppose $(\mathcal{F}_M,\partial^M)$ is a minimal free resolution of $M$. 
We exploit the hypothesized isomorphism on Betti posets $g:B_M\ra B_N$ 
to write a minimal resolution for $N$. 

Since $B_M\cong B_N$, then for every 
$\mathbf{x^{a}}\in B_M$, 
there exists a corresponding monomial $g(\mathbf{x^{a}})\in B_N$ of multidegree $\mathbf{b}$. 
For every such $g(\mathbf{x^a})\in B_N$, write $V_\mathbf{b}^i$ as a rank 
$i$ vector space over $\Bbbk$ with basis $\{e_{\mathbf{b}}^j\}_{j=1}^i$. The 
vector spaces $V_\mathbf{b}^i$ are in one-to-one correspondence 
with the free modules of $\mathcal{F}_M$, with $V_\mathbf{b}^i$ appearing 
in homological degree $r$ since $\redHo_{r-2}(\Delta(1,\mathbf{x^{a}}))$ is nonzero. 

To define the differential of $\mathcal{F}_N$, suppose the differential 
of $\mathcal{F}_M$ takes the form 
$$
\ds\partial^M(e_{\mathbf{a}})=
\sum c_{\mathbf{a},\mathbf{a'}}\cdot\frac{\mathbf{x^{a}}}{\mathbf{x^{a'}}}\cdot e_{\mathbf{a'}}
$$
where $ c_{\mathbf{a},\mathbf{a'}}\in\Bbbk$, 
the bases $e_{\mathbf{a}}$ and $e_{\mathbf{a'}}$ are 
respective generators for free modules in homological degree $r$ and $r-1$, 
and $\mathbf{x^{a'}}<\mathbf{x^a}\in B_M$. 
The differential of the complex $\mathcal{F}_M\otimes R/(x_1-1,\ldots, x_d-1)$ is 
$$
\ds\delta^M(e_{\mathbf{a}})=
\sum c_{\mathbf{a},\mathbf{a'}}\cdot e_{\mathbf{a'}}. 
$$

We now relabel this complex and its differential using the data of the Betti 
poset $B_N$. For every $g(\mathbf{x^a})\in B_N$ of multidegree $\mathbf{b}$, 
define a free module with appropriate shift in multidegree,  $S(-\mathbf{b})$. 
Certainly, each summand of the vector space $V_\mathbf{b}^i$ appearing in our complex 
corresponds to exactly one of the free modules whose shift is $-\mathbf{b}$.
Next, define the action of the differential of $\mathcal{F}_N$ as 
$$
\ds\partial^N(e_{\mathbf{b}})=
\sum c_{\mathbf{a},\mathbf{a'}}\cdot\frac{g(\mathbf{x^{a}})}{g(\mathbf{x^{a'}})}\cdot e_{\mathbf{b'}}
$$
where $\mathbf{{b'}},\mathbf{b}$ are the multidegrees of the monomials corresponding to 
the comparison $g(\mathbf{x^{a'}})<g(\mathbf{x^a})\in B_N$. 

The sequence $\mathcal{F}_N$ certainly has the minimum number of free modules, 
since they were created using the data of the Betti poset $B_N$. Furthermore, 
$\mathcal{F}_N$ is a resolution of $S/N$ if and 
only if the subcomplex of $\mathcal{F}_N(\le m)$ is exact for every $m\in L_N$. 
The exactness of $\mathbb{Z}^t$-graded strands is certainly satisfied, since the 
$\mathbb{Z}^t$-degrees of $\mathcal{F}_N$ are in one-to-one correspondence 
with those of $\mathcal{F}_M$ through the isomorphism $B_M\cong B_N$. 
Hence, $\mathcal{F}_N$ is a minimal free resolution of $S/N$. 
\end{proof}

\begin{remark}
The conclusion of Theorem \ref{BettiMFR} is different than part (3) of Theorem 4.3 in 
\cite{PeevaVelasco}, which also addresses minimal resolutions. 
Their result assumes both an equality of total Betti numbers and 
the existence of a map between lcm-lattices which is 
a least common multiple preserving bijection on atoms. 
Under their assumption, one relabels a resolution of the source lattice's 
monomial ideal to create a resolution of the target lattice's monomial ideal. 

Our assumption of an isomorphism on Betti posets is more restrictive than one 
assuming equality of total Betti numbers. Indeed, ideals with the same total Betti numbers 
can have distinct Betti posets. However, our stronger assumption 
removes the need for a reduction map between lcm-lattices, while maintaining 
the ability to relabel minimal resolutions. 
In particular, monomial ideals which share the same Betti poset need not have 
lcm-lattices which are related by a reduction map. 
\end{remark}

\begin{example}\label{example1}
To illustrate the previous remark consider the following example.  Let
$$M = (b^2ce^2f^2, cde^2f^2, ade^2f^2, abef, ab^2cdf, ab^2cde)$$ and let
$$N = (bce^2f^2, cde^2f^2, are^2f^2, a^2be^2f^2, a^2bcdf, a^2bcde)$$ be
two ideals in $k[a,b,c,d,e,f]$.  Their lcm-lattices (with the Betti 
poset elements marked as bold dots) are shown in Figure \ref{figure1}.

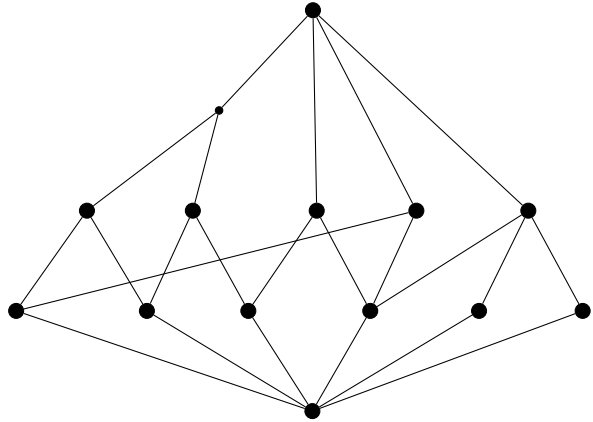
\begin{figure}\label{figure1}
\caption{Lattices from Example \ref{example1}}
\centering
\subfigure[$L_M$ with $B_M$ marked]{
\begin{tikzpicture}[scale=1, vertices/.style={draw, fill=black,
    circle, inner sep=1pt},
  vertices2/.style={draw,fill=black,circle,inner sep=2pt}]
              \node [vertices2] (0) at (-0+.240589,0){};
              \node [vertices2] (1) at (-3.75+.052275,1.33333){};
              \node [vertices2] (2) at (-3.75+1.79047,1.33333){};
              \node [vertices2] (3) at (-3.75+3.13899,1.33333){};
              \node [vertices2] (4) at (-3.75+4.75969,1.33333){};
              \node [vertices2] (5) at (-3.75+6.20612,1.33333){};
              \node [vertices2] (6) at (-3.75+7.58357,1.33333){};
              \node [vertices2] (7) at (-3+.243486,2.66667){};
              \node [vertices2] (8) at (-3+1.65119,2.66667){};
              \node [vertices2] (9) at (-3+3.29794,2.66667){};
              \node [vertices2] (10) at (-3+4.62288,2.66667){};
              \node [vertices2] (11) at (-3+6.11124,2.66667){};
              \node [vertices] (12) at (-1,4){};
              \node [vertices2] (13) at (-0+.248034,5.33333){};
      \foreach \to/\from in {0/1, 0/2, 0/3, 0/4, 0/5, 0/6, 1/10, 1/7, 2/8, 2/7, 3/8, 3/9, 4/9, 4/10, 4/11, 5/11, 6/11, 7/12, 8/12, 9/13, 10/13, 11/13, 12/13}
      \draw [-] (\to)--(\from);
      \end{tikzpicture}}

\subfigure[$L_N$ with $B_N$ marked]{
\begin{tikzpicture}[scale=1, vertices/.style={draw, fill=black,
    circle, inner sep=1pt},
vertices2/.style={draw,fill=black,circle,inner sep=2pt}]
              \node [vertices2] (0) at (-0+.0954296,0){};
              \node [vertices2] (1) at (-3.75+.222501,1.33333){};
              \node [vertices2] (2) at (-3.75+1.70917,1.33333){};
              \node [vertices2] (3) at (-3.75+3.09025,1.33333){};
              \node [vertices2] (4) at (-3.75+4.5988,1.33333){};
              \node [vertices2] (5) at (-3.75+6.06646,1.33333){};
              \node [vertices2] (6) at (-3.75+7.70945,1.33333){};
              \node [vertices2] (7) at (-3+.139087,2.66667){};
              \node [vertices2] (8) at (-3+1.70673,2.66667){};
              \node [vertices2] (9) at (-3+3.05989,2.66667){};
              \node [vertices2] (10) at (-3+4.79619,2.66667){};
              \node [vertices2] (11) at (-3+6.16729,2.66667){};
              \node [vertices] (12) at (-0+.158475,4){};
              \node [vertices2] (13) at (-0+.0718432,5.33333){};
      \foreach \to/\from in {0/1, 0/2, 0/3, 0/4, 0/5, 0/6, 1/10, 1/7, 2/8, 2/7, 3/8, 3/9, 4/9, 4/10, 4/11, 5/11, 6/11, 7/13, 8/12, 9/12, 10/13, 11/13, 12/13}
      \draw [-] (\to)--(\from);
      \end{tikzpicture}
}
\end{figure}

Note that $B_M$ is isomorphic to $B_N$, but there is no join
preserving map between $L_M$ and $L_N$. The labeled
elements which are not in the respective Betti posets in both figures
are incompatible. 
\end{example}

In \cite{CMRigid}, we study a natural generalization of generic ideals, the class of rigid ideals. 
A rigid ideal can have non-simplicial multidegrees which correspond to unique multigraded 
basis elements in the minimal free resolution. To proceed, recall the 
class of rigid monomial ideals. 

\begin{definition}\cite{CMRigid}
Let $M$ be a monomial ideal, with lcm-lattice $L_M$. 
Then $M$ is a \emph{rigid ideal} if the following two conditions hold: 
\begin{itemize}
\item[(R1)] For every $p\in L_M$, we have $h_i(\Delta_p)=1$ 
	for at most one $i$.
\item[(R2)] If there exist $p,q \in L_M$, where 
         $h_i(\Delta_p)=h_i(\Delta_{q})=1$ 
         for some $i$ then $p$ and $q$ are incomparable in $L_M$. 
\end{itemize}
\end{definition}

As one might expect, the rigidity of a monomial ideal is dependent on the characteristic 
of the field $\Bbbk$.  
 
\begin{remark}\label{stratification}
One implication of the definition of rigidity is
that if $\mathbf{b}$ and $\mathbf{b'}$ are such that
$\beta_{i,\mathbf{b}}$ and $\beta_{j,\mathbf{b'}}$ are nonzero and
$\mathbf{b} > \mathbf{b'}$ in $L$ then $i > j$.  To see this assume $j
> i$. Since $\mathbf{b'}$ is the $\mathbb{Z}^d$-degree of a $j$th
syzygy then there must be a $(j-1)$st syzygy whose 
$\mathbb{Z}^d$-degree divides $\mathbf{b'}$ (i.e. it will be less than
$\mathbf{b'}$ in $L$).  Repeating this we obtain a chain of elements
in $L$ which ends in an element corresponding to an $i$th syzygy of
$\mathbb{Z}^d$-degree $\mathbf{c} < \mathbf{b'}$ in $L$.  
This contradicts rigidity condition (R2), since by transitivity 
$\mathbf{c} < \mathbf{b}$ in $L$ and both $\beta_{i,\mathbf{b}}$ 
and $\beta_{i, \mathbf{c}}$ are nonzero.  Hence, $i > j$.
\end{remark}

We are in a position to state the second commutative algebra result of this paper. 

\begin{theorem}\label{BettiSupportsRigid}
The Betti poset supports the minimal free resolution of a rigid monomial ideal.  
\end{theorem}

The proof of Theorem \ref{BettiSupportsRigid} amounts to showing that a sequence 
of vector spaces 
derived from the combinatorial structure of the Betti poset $B_M$ is an exact complex. 
We postpone the details of this technical argument to Section \ref{S:MinResRigid}. 

In the remainder of this section, we discuss the significance 
of Theorem \ref{BettiSupportsRigid} with regards to a 
progression of ideas concerning the minimal resolution of monomial ideals. 
We use examples to make the case for the 
importance of the class of rigid ideals as fundamental to the construction of 
minimal resolutions. 

\begin{remark}
Note that one can always modify the minimal resolution 
of an arbitrary ideal by scaling 
or permuting the $\mathbb{Z}^d$-graded basis vectors and 
propagating this change in order to preserve exactness. Furthermore, for non-rigid ideals 
one may also change the $\mathbb{Z}^d$-graded basis in ways which change the (combinatorial) 
structure of the minimal free resolution. In the case of rigid ideals, such a 
change is impossible. In particular, we characterize 
rigid ideals in \cite[Proposition 1.5]{CMRigid} as having a minimal free resolution with 
unique $\mathbb{Z}^d$-graded basis. Theorem \ref{BettiSupportsRigid} 
therefore provides a unique combinatorial object which encodes both the unique $\mathbb{Z}^d$-graded 
basis and the mapping structure of the minimal free resolution of a rigid ideal. 
\end{remark}

Theorem \ref{BettiSupportsRigid}'s combinatorial prescription for the minimal resolution of 
a rigid monomial ideal allows us to take aim at Kaplansky's original question. 
For a non-rigid monomial ideal $M$ we strive to find a \emph{rigid deformation} of $M$ 
whose resolution can be relabeled to give a resolution of $M$.  We use the word deformation here 
as a reference to the notion of a \emph{generic deformation} in \cite{BPS} and \cite{MSY}.  
Our notion of producing a rigid deformation will not involve perturbing the 
exponents of the ideal's generators, so we omit an explicit discussion of 
deformations of exponents as mentioned in \cite{BPS} and \cite{MSY}. 
Instead we will discuss the equivalent notion found in \cite{GPW} and \cite{mapes}.  
First we will need the following definition. 

\begin{definition}\cite{phan}
Let $\mathcal{L}(n)$ be the set of all finite atomic lattices with $n$ ordered atoms.  
Set $P\geq Q\in\LN$ if there exists a join preserving map $f:P \rightarrow Q$ which is a bijection on atoms.  
\end{definition}

The condition that there be a join preserving map which is a bijection on atoms, is the same condition found in Theorem 3.3
of \cite{GPW} which states that a minimal resolution of an ideal with lcm-lattice $P$ can be relabeled to be a resolution of an 
ideal with lcm-lattice $Q$.  In \cite{GPW} the authors note that the deformation of exponents from \cite{BPS} produces 
lcm-lattices with this join preserving map.  Moreover, \cite[Theorem 5.1]{mapes} states that for any two comparable 
finite atomic lattices in $\mathcal{L}(n)$, there exist monomial ideals so that the ideal whose lcm-lattice is $P$ is a 
deformation of exponents of the ideal whose lcm-lattice is $Q$.  

\begin{definition}
A monomial ideal $J$ is a \emph{rigid deformation} of the monomial ideal 
$I$ if $J$ is a rigid ideal, and the resolution of $J$ can be relabeled to minimally resolve $I$. 
\end{definition} 

Note here we do not simply require comparability in $\mathcal{L}(n)$
as would be suggested by the work of \cite{GPW} and \cite{mapes}.
Corollary \ref{goDownToGoRigid} and Example \ref{noRigidAbove} should
illuminate to the reader why we chose to make the definition this way.

\begin{corollary} \label{goDownToGoRigid}
Let $L$ be a finite atomic lattice and suppose that $P>L\in\LN$ is rigid with the same Betti numbers as $L$. 
Then every $L'$ for which $B(L)\cong B(L')$ has its minimal resolution supported on $B(P)$. 
In other words, $P$ is a rigid deformation of $L'$. 
\end{corollary}

\begin{example}\label{noRigidAbove}
Let $R=\Bbbk[a,s,t,u,v,w,x,y,z]$ and consider the squarefree monomial ideal 
$I=(uvxyz,atwxyz,stuwz,astuvwx,asuvwxy,stvyz)$ 
whose lcm-lattice is pictured in Figure 2. 
The rigid deformation of $I$ has an lcm-lattice which is not comparable to $L_I$ in $\mathcal{L}(6)$, 
although the minimal free resolution of $I$ is supported on a regular
CW-complex with the intersection property.  In particular, the
highlighted element in $L_I$ is not in $B_I$. Any attempt at a rigid 
deformation that produces a lattice comparable to $L_I$ forces an 
increase in total Betti numbers.  To produce a rigid
deformation of $I$, we locate a lattice which is comparable to $B_I$, the lattice 
created from $L_I$ by removing the highlighted element removed. 
We are fortunate that $B_I$ is also a finite atomic lattice. 

\center
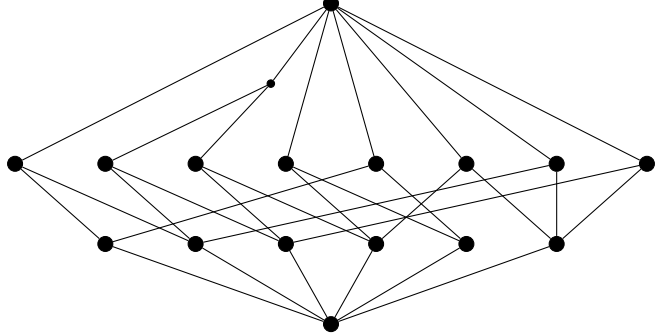
\begin{figure}\label{NoHigherRigid}
\centering
\caption{The lcm-lattice $L_I$ of Example \ref{noRigidAbove}}
 \begin{tikzpicture}[scale=0.8, vertices/.style={draw, fill=black,
    circle, inner sep=1pt},
  vertices2/.style={draw,fill=black,circle,inner sep=2pt}]
             \node [vertices2] (0) at (-0+0,0){};
             \node [vertices2] (9) at (-3.75+0,1.33333){};
             \node [vertices2] (1) at (-3.75+1.5,1.33333){};
             \node [vertices2] (2) at (-3.75+3,1.33333){};
             \node [vertices2] (3) at (-3.75+4.5,1.33333){};
             \node [vertices2] (13) at (-3.75+6,1.33333){};
             \node [vertices2] (7) at (-3.75+7.5,1.33333){};
             \node [vertices2] (10) at (-5.25+0,2.66667){};
             \node [vertices2] (4) at (-5.25+1.5,2.66667){};
             \node [vertices2] (5) at (-5.25+3,2.66667){};
             \node [vertices2] (14) at (-5.25+4.5,2.66667){};
             \node [vertices2] (15) at (-5.25+6,2.66667){};
             \node [vertices2] (12) at (-5.25+7.5,2.66667){};
             \node [vertices2] (8) at (-5.25+9,2.66667){};
             \node [vertices2] (11) at (-5.25+10.5,2.66667){};
             \node [vertices] (6) at (-1,4){};
             \node [vertices2] (16) at (-0+0,5.33333){};
     \foreach \to/\from in {0/9, 0/1, 0/2, 0/3, 0/13, 0/7, 1/4, 1/8, 1/10, 2/4, 2/5, 2/11, 3/12, 3/5, 3/14, 4/6, 5/6, 6/16, 7/12, 7/8, 7/11, 8/16, 9/10, 9/15, 10/16, 11/16, 12/16, 13/14, 13/15, 14/16, 15/16}
     \draw [-] (\to)--(\from);
     \end{tikzpicture}
\end{figure}
\end{example}

We have the following result guaranteeing the existence of a rigid deformation for certain monomial ideals. 

\begin{proposition}\label{simplicialRigidDefo}
If $I$ is a monomial ideal whose minimal free resolution is supported on a simplicial complex $X$, 
then there exists a rigid ideal $J$ whose minimal free resolution 
is also supported on $X$ such that $L_I<L_J\in\LN$. 
That is, $J$ is a simplicial rigid deformation of $I$. 
\end{proposition}  

\begin{proof}
For ease of notation let us establish the following, $L = L_I$, and
$P$ is the augmented face poset of simplices of $X$, making 
$P$ a finite atomic lattice.  Note that each of these
lattices have the same number of atoms, and that in
terms of atomic supports, certain elements of $L$ correspond to
elements in $P$.  In what follows we will denote elements of $L$ which
do not also correspond to elements in $P$ as
$l$ (possibly indexed), elements of $P$ as $p$ (possibly indexed).
Our goal is to construct a new lattice $T$ with the
properties that $T$ is greater than $L$ and $P$ in $\mathcal{L}(n)$
such that there is an equality of total Betti numbers $\beta(T) = \beta(P)$.  
This lattice $T$ will give rise to the monomial ideal $J$ (and in fact $L_J$ will be isomorphic to $T$).  

Thinking of our finite atomic lattices as sets of sets (where the sets
correspond to atomic supports of each of the elements in the lattice),
let $T$ be the meet closure of $L \cup P$. By construction, $T$ is a finite
atomic lattice.  Moreover, we can think of $T$ as
consisting of elements of the following type 
$\{l \,\mid l \in L \mbox{ and does not correspond to an element of $P$}\}$ 
and $\{p \,\mid p \in P\}$.  
Note that one might worry that there is a subset of elements in $T$
which takes the form $\{m \,\mid m = l \wedge p \}$.  
This is however, not the case. Since $X$ is a simplicial complex, for each 
$p \in P$ every subset of the atomic support of $p$ in $P$ corresponds
to a distinct element in $P$.          

We need to show the following.\\
\begin{enumerate}
\item $\tilde{h}_i (\Delta_l^T) = 0$ for all $i$.
\item $\tilde{h}_i (\Delta_p^T) = \tilde{h}_i(\Delta_p^P)$ for all $i$.\\
\end{enumerate}

Proving the second item is easy. Since $X$ is a simplicial complex, none
of the elements $l$ can be less than any element $p$.  This means all of the
open intervals $(\hat{0},p)$ in $T$ are isomoprphic to the same
interval in $P$, guaranteeing that the homology groups are the same.

To prove the first item, we first assume $l$ is not greater than any 
other elements of type $l$.  In other words, all of the elements less 
than $l$ are of the type $p$.  In this case the open interval $(\hat{0},l)$ 
in $T$ is isomorphic to the union of half closed intervals 
\[P_{\leq l} = \bigcup_{p_i \ld_T l}(\hat{0},p_i]_P.\]  
By the acyclicity condition in \cite{BS}, 
we know that $X_{\leq l}$ is acyclic for every $l \in L$.  Since 
$\Delta(P_{\leq l})$ is the barycentric subdivision of $X_{\leq l}$ we  
conclude that $\tilde{h}_i (\Delta_l^T) = 0$ for all $i$. 

Now we need to remove the assumption that $l$ is not greater than any 
elements of the type $l$.  We do this by working 
up from the bottom of the lattice $T$.  We want to find an $l$ 
satisfying the earlier assumption, by the above argument 
$\tilde{h}_i(\Delta_l^T) = 0$ for all $i$. 
By Lemma \ref{removeOne} we can delete $l$ to create a subposet $T'$ so any $l' > l$ in $T$ 
no longer has $l$ below them in $T'$ and the homology computations for intervals in $T'$ are the same 
as in $T$.  Now if $l$ was the only element of type $l$ less than $l'$ 
in $T$, then in $T'$ the element $l'$ satisfies the assumption that it is only 
greater than elements of type $p$. Hence, the previous argument applies so that 
$\tilde{h}_i(\Delta_l^T)= \tilde{h}_i(\Delta_l^{T'}) = 0$ for all $i$.  
Note that in iterating this process we reduce $T$ down to $P$
(or $P - \{\hat{1}\}$ if $X$ is not a simplex), guaranteeing that $\beta(T) = \beta(P)$.  
\end{proof}

We have reason to believe that a more general statement is
true and propose the following conjecture. Recall that a CW-complex 
is said to have the intersection property if the intersection of any two cells 
is also a cell. 

\begin{conjecture}\label{CWIntConjecture}
If $I$ is a monomial ideal with a minimal resolution supported on a
regular CW-complex with the intersection property then $I$ admits a rigid deformation. 
\end{conjecture}

The notion of rigidity is not limited to resolutions supported on topological objects. 
In fact, Velasco's example of an ideal with non-CW resolution \cite{Velasco} is a rigid ideal. 
However, the assumption in Conjecture \ref{CWIntConjecture} that the regular 
CW-complex satisfy the intersection property is necessary. Consider the 
following example, whose homological structure was pointed out to us by Adam Boocher. 

\begin{example} 
Let $I$ be  the edge ideal of the hexagon, 
$$I=(x_1x_2,\, x_2x_3,\, x_3x_4,\, x_4x_5,\, x_5x_6,\, x_1x_6).$$
This ideal's resolution is supported on a three-dimensional regular CW-complex 
whose $f$-vector is $(1,6,9,6,2)$. 
This complex is absent the intersection property. 
Direct calculation shows that adding any single element to 
the lcm-lattice $L_I$ increases Betti numbers. Hence, there is no finite atomic lattice with the same 
Betti numbers whose Betti poset is isomorphic to $B_I$. 
Therefore, $I$ does not admit a rigid deformation. 

In this example, the key step which does not allow us to proceed along
the lines of the previous proof is that for the CW-complex supporting
the resolution, the meet-closure of the poset of cells is isomorphic to $L_I$. 
Hence, no distinct lattice $P$ exists, and we cannot construct the corresponding lattice $T$.  
\end{example}

\section{The minimal resolution of a rigid ideal}\label{S:MinResRigid}

In order to prove Theorem \ref{BettiSupportsRigid}, 
we first describe the construction found in \cite{Clark}, which 
produces a sequence of vector spaces and maps using the data 
of a finite poset $P$. For $\ell\ge 1$, the vector spaces in this sequence are 
$$
\ds \mathcal{D}_\ell
=\bigoplus_{p \in P\smsm\{\hat{0}\}}\redHo_{\ell-2}(\Delta_p,\Bbbk).
$$
In the border case, $\mathcal{D}_0=\redHo_{-1}(\{\varnothing\},\Bbbk)$, 
a one-dimensional $\Bbbk$-vector space.  
Note that the atoms (level one elements) of $P$ 
index the nontrivial components of the vector space $\mathcal{D}_1$.

These vector spaces are yoked into a sequence
$$
\mathcal{D}(P):
\cdots\ra
\mathcal{D}_\ell\stackrel{\varphi_{\ell}}{\lra}
\mathcal{D}_{\ell-1}\ra\cdots\ra
\mathcal{D}_1\stackrel{\varphi_{1}}{\lra}
\mathcal{D}_0, 
$$ 
by maps $\varphi_i$ whose structure comes from simplicial topology. 
To be precise, denote the order complex of a  
half closed interval $(\hat{0},p]$ as $\Delta_{\cl p}$. We have 
the following decomposition of the order complex of the open interval $(\hat{0},q)$
\[\Delta_q = \bigcup_{p \ld q} \Delta_{\cl{p}}.\]  

Focusing on the homological interaction between a fixed $\Delta_{\cl{p}}$ 
and the rest of $\Delta_q$, set 
\[
\Delta_{q,p} 
= \Delta_{\cl{p}} \cap \left(\bigcup_{\underset{b \neq p}{\overset{b \ld q}{}}} \Delta_{\cl{b}}\right).
\] 
and note that $\Delta_{q,p} \subset \Delta_p$.  

The maps $\varphi_i$ are defined componentwise for every $p$ using the 
connecting maps in the Mayer-Vietoris sequence for the triple 
\[
\left(\Delta_{\cl{p}},\,\,\, \bigcup_{\underset{b \neq p}{\overset{b \ld q}{}}} \Delta_{\cl{b}},\,\,\, \Delta_q\right).
\]  
Specifically, $\varphi_i: \mathcal{D}_i \rightarrow \mathcal{D}_{i-1}$ is defined as 
$\ds\varphi_i \vert_{\mathcal{D}_i, q} = \sum_{p \ld q} \varphi_i^{q,p}$. 
The component map 
$\varphi_i^{q,p} = \iota \circ \delta_{i-2}^{q,p}:\mathcal{D}_{i,q} \rightarrow \mathcal{D}_{i-1,p}$, 
is the composition of the map on homology induced by inclusion, 
$\iota: \tilde{H}_{i-3}(\Delta_{q,p},\Bbbk) \rightarrow \tilde{H}_{i-3}(\Delta_p, \Bbbk)$ with 
$\delta_{i-2}^{q,p}: \tilde{H}_{i-2}(\Delta_q,\Bbbk) \rightarrow \tilde{H}_{i-3}(\Delta_{q,p},\Bbbk)$, 
the connecting map from the Mayer-Vietoris sequence.

We proceed with the proof by first showing that $\mathcal{D}(B_M)$ is a complex when $M$ is
rigid.  

For a poset element $q\in P$, write $\ell(q)=\max\{j:p_1\lessdot p_2 \lessdot\ldots\lessdot p_j=q\}$ for the level of $q$, 
the maximum possible length of a saturated chain ending in $q$. As a consequence, $\ell(\hat{0})=0$ and 
the atoms of $P$ are of level 1. The rank of $q$ is $\rank(q)=1+\ell(q)$ and may be thought of as the maximum 
number of poset elements (including $q$) which appear within a saturated chain which ends at $q$. 
For $q\in P$, write $\mathcal{D}(P)_{\leq q}$ 
for the subsequence of $\mathcal{D}(P)$ constructed 
by using the half-closed interval $(\hat{0},q]$ with maps given as  
restrictions of the maps from $\mathcal{D}(P)$. 

\begin{definition}
Given $q \in P$ let $C$ be the set of all chains $c \in (\hat{0}, q]$ 
such that $c$ has $\ell(q)$ elements.  
The \emph{maximal ranked subposet of} $(\hat{0}, q]$ 
is the set $$\maxRanked (P,q)=\{p\in (\hat{0}, q] : p \in c \in C\},$$ 
with comparison inherited from $P$. 
\end{definition} 

\begin{remark}\label{someAreRanked}
By definition, $\maxRanked(B_M,q)$ is ranked for all $q \in B_M$.  Using 
Proposition 7.1 from the Appendix of \cite{Clark}, the sequence $\mathcal{D}(\maxRanked(B_M, q))$ is a complex.  
Moreover, in the case when $B_M$ itself is ranked, then by the same proposition, $\mathcal{D}(B_M)$ is a complex.  
\end{remark}

\begin{remark}\label{stratification2}
Note that Remark \ref{stratification} implies that for a rigid monomial ideal, 
if one records the number $i$ for which $\beta_{i,p}$ is nonzero at each
element $p$ in $\maxRanked(B_M,q)$ then the $i$ strictly
increase along chains.  Specifically, $i=0$ at the minimal element, and $i$ 
increases by one, traveling cover by cover along chains in $\maxRanked(B_M,q)$.
\end{remark}

Our goal is to use the fact in the previous remark to show that 
$\mathcal{D}(B_M)$ is a complex even when $B_M$ is not ranked.  We aim
to do this by showing that the last two maps in any multigraded strand
$\mathcal{D}(B_M)_{\leq q}$ behave exactly like the maps coming from a
ranked poset.  The following lemma shows that the free modules we
obtain by passing to $\maxRanked(B_M,q)$ will agree in the last two
spots of the complex.

\begin{lemma}\label{freeModsSameInMR}
Let $M$ is rigid and $\tilde{h}_t(\Delta_q^{B_M},k) = 1$ for a specific
$t$. For $p \in \maxRanked(B_M,q)$ such that
$\tilde{h}_i(\Delta_p^{B_M},k) = 1$ for $i = t, t-1,$ or $ t-2$
we obtain the equivalence $$\tilde{h}_i(\Delta_p^{B_M},k) =
\tilde{h}_i(\Delta_p^{\maxRanked(B_M,q)},k)$$ for all $i > t-2$.
\end{lemma}

\begin{proof}
If $p \in \maxRanked(B_M,q)$ then $p \in B_M$.  By rigidity of $M$,
let us say that $\tilde{h}_i(\Delta_p^{B_M},k) = 1$ and
$\tilde{h}_j(\Delta_p^{B_M},k) = 0$ for $j \not=i$.  This means that
the maximal chains in $(\hat{0}, p)_{B_M}$ are of length $i+1$ and
correspond to the faces forming a cycle $w_p$ which generates
$\tilde{H}_i(\Delta_p^{B_M},k)$.  As these are maximal chains and
since $p \in \maxRanked(B_M,q)$ we can see that these are the same as
the maximal chains in $(\hat{0},p)_{\maxRanked(B_M,q)}$.  Thus $w_p$
also generates homology for
$\tilde{H}_i(\Delta_p^{\maxRanked(B_M,q)},k)$.  

We will use the Mayer-Vietoris sequence, and for simplicity designate
$X = \Delta_p^{B_M}$, $A = \Delta_p^{\maxRanked(B_M,q)}$, and define
the poset $B' = \{r \in B_M \,\vert\, r \not\in \maxRanked(B_M,q),
r<p\}$ and let $B = \Delta(B')$.  By construction we see that $X = A
\cup B$.  By rigidity, we know that $\tilde{H
}_i (X,k)$ is nonzero for
  only one $i$ (which by the assumptions of the lemma is either $t,
  t-1,$ or $t-2$) and is zero for all other $i$. Moreover, we can
  also see that by rigidity that $\tilde{H}_i(B,k)$ is zero for all $i$,
  since if it were not (say it was nonzero for some $j$) then
  $\tilde{H}_i (X,k)$ would also be nonzero in homological $j$ which
    would be a contraditction (see the definition of rigidity and
    Corollary 1.3 in \cite{CMRigid}).  Thus by Mayer-Vietoris we get the
    following isomporhisms from the long exact sequence:
\begin{enumerate}
\item $\tilde{H}_i(A, k) \cong \tilde{H}_i(X,k)$ because
  $\tilde{H}_i(A\cap B, k)$ is necessiarly zero since the dimension of
  $A \cap B$ is strictly less than the dimension of $B$ which is less
  than $i$.
\item $\tilde{H}_j(A\cap B, k) \cong \tilde{H}_j (A, k)$ for $j \neq i$
  since $\tilde{H}_j(X,k)$ is zero for all $j\neq i$. 
\end{enumerate}     

In the second case we want to show that these groups are both zero for
$j = t$, or $ t-1$.  For $j > i$, both are clearly zero since the
dimension of $A$ is $i$ (regardless of how $i$ relates to $t$).  For $j < i$, we need only show that the
isomporhism holds for $j = t$, or  $t-1$.  
Here the maximal chains of $B'$ are of at most length $i$ since chains
in $(\hat{0},p)^B_M$ are of at most length $i+1$.  Moreover as
the maximal chains of $B'$ are not maximal chains of
$\maxRanked(B_M,q)$ and since $p \not\in B'$, this further limits the
maximal length of chains in $B'$ to be of at most length $i-1$.  Thus
in the simpicial complex $A\cup B$ the maximum dimension of any face
is $i-2$.  Now consider our limits on what values $i$ can take.  If $i
= t$, then the dimension of $A\cup B$ is $t-2$  and the homology
vanishes in homolgial degrees $t$ and $t-1$.  If $i = t-1$, the
dimension of $A \cup B$ is $t-3$ and so the homology vanishes in the
appropriate places.  And the same is true if $i = t-2$.  
\end{proof}

The next lemma demonstrates certain components of the maps in $\mathcal{D}(B_M)_{\leq q}$ are zero.

\begin{lemma}\label{mapsAreZero}
If $M$ is rigid and $p \in B_M$ is covered by $q$, and $p \not\in \maxRanked(B_M,q)$
then $\varphi_i^{q,p}$ in $\mathcal{D}(B_M)$ is the zero map for all $i$.
\end{lemma}

\begin{proof}
Because $p,q \in B_M$ we know that there exist $k,l \in \mathbb{Z}$ such that 
$\tilde{h}_k (\Delta_p, \Bbbk) = 1$ and $\tilde{h}_l (\Delta_q, \Bbbk) = 1$.  

First we need to show that if $p \ld q$ in $B_M$ and $p \not\in \maxRanked(B_M,q)$ then $l-k > 1$.  
Clearly $l-k \not= 0$ for if so, we would contradict rigidity
condition (R2).  Moreover by Remark \ref{stratification}, (R2) also implies that $l > k$. 

Given that $l > k$, it remains to show that $l > k + 1$.  We 
see that the subposet $(\hat{0},q]$ in $\maxRanked(B_M,q)$ must
contain a maximal chain of length $l+1$ (in fact all chains in
$\maxRanked(B_M,q)$ have the same length which is $l+1$).
So the elements $p \in B_M$ which are not elements of $\maxRanked(B_M,q)$ must lie in a chain of largest possible length 
which is less than $l+1$. Hence, $(\hat{0}, p]$ must contain a maximal chain of length $k+1$. 
Since $p \ld q$, the maximum length of a chain in $(\hat{0},q]$ which contains $p$ must be $k + 2$.  
Thus, $k+2 < l +1$, or equivalently $k + 1 < l$, as claimed. 

Now we can finally show that $\varphi_i^{q,p} = 0$.  We will do this by showing that 
$$\delta_{i-2}^{q,p}: \tilde{H}_{i-2}(\Delta_q, \Bbbk) \rightarrow \tilde{H}_{i-3}(\Delta_{q,p}, \Bbbk)$$ 
is zero for all $i$.  Since $M$ is rigid, if $\Delta_q$ only has nonzero homology 
in homological degree $l$ then we need only focus 
our attention on the case when $i-2 = l$ since otherwise the map is already zero.  Specifically, we must 
show that $\tilde{H}_{l-1} (\Delta_{q,p}, \Bbbk) = 0$.    

First, note that $\Delta_{q,p}$ can also be expressed as the order complex of the poset 
\[ (B_M)_{q,p}: = (\hat{0}, p] \cap \left(\bigcup_{\underset{p' \neq p}{\overset{p' \ld q}{}}} (\hat{0}, p'] \right).\]

Necessarily, $(B_M)_{q,p} \subset (\hat{0},p)$ because $p \not\in (\hat{0},p')$ for any such $p' \ld q$.  
This means that the length of the longest chain in $(B_M)_{q,p}$ is less than or equal to $k$ and 
furthermore, that the maximum possible dimension of a face in $\Delta_{q,p}$ is $k-1$.  So if $i > k-1$ it must be that 
$\tilde{h}_i (\Delta_{q,p}, \Bbbk) \not= 0$.  As previously argued, we know that $l >  k+1$ so 
that $k-1 < k  < l-1$.  Thus $\tilde{H}_{l-1} (\Delta_{q,p}, \Bbbk) = 0$.      
\end{proof}

Now we will show that the last two maps in $\mathcal{D}(B_M)_{\leq q}$
and $\mathcal{D}(\maxRanked(B_M, q))$ are the same.

\begin{lemma}\label{mapsActLikeRankedCase}  
The length of $\mathcal{D}(B_M)_{\leq q}$ is equal to the length of 
$\mathcal{D}(\maxRanked(B_M, q))$. Moreover, writing this common 
length as $l$, the first homological degree where the 
two sequences can possibly differ is $l-2$.
\end{lemma}

\begin{proof}
Denote $\maxRanked(B_M,q) - \{q\}$ as $\maxRanked(B_M, \hat{q})$ and write 
$\Delta_{\maxRanked,q}$ for the order 
complex $\Delta(\maxRanked(B_M, \hat{q}))$.  To see that the 
length of both sequences is the positive integer $l$, it suffices to show that 
$$\tilde{h}_{l-2}(\Delta_{\maxRanked,q}, \Bbbk) = \tilde{h}_{l-2}(\Delta_q, \Bbbk).$$

This follows immediately from the fact that the maximal chains of the ranked poset 
$\maxRanked(B_M, \hat{q})$ are exactly the maximal chains of $(\hat{0}, q)$, 
and that these chains correspond to the $(l-2)$-faces of the respective order complexes.

To see that the last two maps are the same in the sequences
$\mathcal{D}(B_M)_{\leq q}$  and $\mathcal{D}(\maxRanked(B_M, q))$
first observe that if $B_M = \maxRanked(B_M,q)$ then the two sequences
of maps are identical.  Thus it remains to consider the case when  $(\hat{0}, q] \subset B_M$ is not ranked.  
Lemma \ref{mapsAreZero} guarantees that $\varphi_l$ is identical in both
sequences.  Moreover, 
we know that the basis elements of $\mathcal{D}_{l-1}$ (in the
sequence $\mathcal{D}(B_M)_{\leq q}$) must correspond to elements 
$p\in \maxRanked(B_M,q)$ where $p \ld q$.  Indeed if this were not the
case consider an element $p'$
covered by $q$ in $B_M$ which is not in $\maxRanked(B_M,q)$. By Remark \ref{stratification2}, 
the Betti number $\beta_{i,p'}$ must be nonzero for some 
$i <l-1$ since $p'$ is in a chain of length less than $l+1$.  Thus, the free modules in position 
$l-1$ are the same in both sequences. 

Finally we will show that the maps $\varphi_{l-1}$ are also the same.
Note by \ref{freeModsSameInMR}, it is possible for the free modules in position $l-2$ to differ
between the two sequences, but the domain for $\varphi_{l-1}$ in
both sequences are the same.  Moreover, if there are basis elements in
position $l-2$ for the sequence $\mathcal{D}(B_M)_{\leq q}$ which are
not in $\mathcal{D}(\maxRanked(B_M, q))$ then they must correspond to
elements $p' \in B_M$ such that $p' \not \in \maxRanked(B_M, q)$.
Since $p'$ is in a non-maximal chain this means that $p'$ is not
comparable to any of the elements $p$ corresponding to basis elements
in position $l-1$ since they are all in $\maxRanked(B_M, q)$.  Thus
$\varphi_{l-1}^{p,p'} = 0 $ for all $p \in \maxRanked(B_M, q)$
corresponding to basis elements in position $l-1$.  Moreover this
implies that the only
elements in $(B_M)_{\leq q}$ contributing nonzero components of
$\varphi_{l-1}$ correspond to elements in $\maxRanked(B_M, q)$.  Thus
the maps $\varphi_{l-1}$ are the same between the two sequences.  
\end{proof}

With these lemmas in hand we now prove that
$\mathcal{D}(B_M)$ is a complex.

\begin{theorem}
$\mathcal{D}(B_M)$ is a complex.
\end{theorem}

\begin{proof}
First note that if $B_M$ is ranked then $\mathcal{D}(B_M)$  is a
complex by Proposition 7.1 in the Appendix of \cite{Clark}.  

Assume that $B_M$ is not ranked.  We need to show that for any element 
$a\in \mathcal{D}_i$, we have $\varphi_{i-1}(\varphi_i(a)) = 0$.  It is
sufficient to check this for basis elements.  Let $q \in B_M$
correspond to a basis element in $\mathcal{D}_i$.  To verify that 
$\varphi_{i-1}(\varphi_i(q)) = 0$, it is sufficient to check the composition 
of the maps in $\mathcal{D}(B_M)_{\leq q}$.  By Lemma
\ref{mapsActLikeRankedCase} we see that this criteria can be 
checked using the corresponding maps in $\mathcal{D}(\maxRanked(B_M, q))$.
However, since $\maxRanked(B_M,q)$ is ranked then, Proposition 7.1
of \cite{Clark} implies that $\mathcal{D}(\maxRanked(B_M, q))$ is a
complex.  Thus $\varphi_{i-1}(\varphi_i(q)) = 0$, as needed.
\end{proof}      

With $\mathcal{D}(B_M)$ shown to be a complex, it remains to verify exactness. 
In the arguments that follow, square brackets written around 
a simplicial chain always denote a homology class. 

\begin{theorem}
If $M$ is a rigid monomial ideal, then $\mathcal{D}(B_M)$ is exact.
\end{theorem}

\begin{proof}
Write $p$ for the projective dimension of the rigid module $R/M$ 
and fix $1\leqslant i\leqslant p$.  
The construction of $\mathcal{D}(B_M)$ guarantees that 
an element $v\in\ker(\varphi_i)$ has the structure of a 
a finite sum $v=\sum c_q\cdot [w_q]$, where 
$q\in B_M$, the coefficient $c_q\in\Bbbk$ and $[w_q]$ 
is a generator for the vector space $\Ho_{i-2}(\Delta_q,\Bbbk)$. 

Since $v$ is an abstract sum of homology classes, each coming 
from a unique vector space summand, $v=0$ 
if and only if $c_q=0$ for every $q$. 
In such a scenario, $v\in\image(\varphi_{i+1})$. 
Therefore, in order to prove exactness, we must show 
that every nonzero kernel element $v$ is an element of $\image(\varphi_{i+1})$. 

Since the sum $v$ is built from the homology classes of 
$\Ho_{i-2}(\Delta^{B_M}_q,\Bbbk)$ for various $q\in B_M$, 
we appeal to the structure of $B_M$ as a subposet 
of $L_M$ and to the relationship between the simplicial homology 
of the relevant open intervals in these posets. In particular, Theorem \ref{removeNC} 
guarantees that the class $[w_q]$ also generates 
$\Ho_{i-2}(\Delta^{L_M}_q,\Bbbk)$. 

As preparation for the rest of the proof, recall the notion of a cone of a simplical complex 
taken over a disjoint apex $a$. This complex consists of simplices 
which are simplicial joins of the apex $a$ with simplices $\sigma$ of 
the original simplicial complex. We write $\{a,\sigma\}$ for the 
associated simplicial chain in a cone complex. 

Define $u=\ds\sum c_q\cdot \{q, w_q\}$ to be the finite sum 
of simplicial chains in the algebraic chain complex $\mathcal{C}(\Delta(B_{M}))$. 
The chains in this sum are created by coning the chains representing the 
classes $[w_q]$ over the apex $q$. The sum $u$ is an oriented simplicial 
chain in $\Delta(B_M)$  and the inclusion of posets $B_M\subseteq L_M$ 
guarantees that our sum is also an oriented chain of $\Delta(L_M)$. 
Since $L_M$ is a lattice, then there exists $y\in L_M$ such that 
$y=\vee q$ for those $q$ appearing in the sum.

\emph{Case 1:} 
Suppose $\Ho_{i-1}(\Delta^{L_M}_y,\Bbbk)\ne0$, so that $y\in B_M$. 
Such a $y$ covers every $q\in B_M$ which appears the 
sum $u$. Indeed, any $x\in B_M$ with the property that 
$q< x< y$ must have $\Ho_{j}(\Delta^{L_M}_x,\Bbbk)\ne0$ 
for $i-2<j<i-1$, which is impossible. 

If the class $\ds [u]=\left[\sum c_q\cdot \{q, w_q\}\right]$ 
generates $\Ho_{i-1}(\Delta^{L_M}_y,\Bbbk)$, then it certainly 
generates  $\Ho_{i-1}(\Delta^{B_M}_y,\Bbbk)$ since $[u]$ is 
defined using simplices of $\Delta(B_N)$. This is not necessarily the case 
if we take an arbitrary generator of the homology of $\Delta^{L_M}_y$. 

Applying the map $\varphi_{i+1}$ directly to $[u]$, we have 
\begin{eqnarray}\label{E:PhiActingOnU}
\varphi_{i+1}([u]) & = & \varphi_{i+1}\left(\left[\sum c_q\cdot \left\{q,w_q\right\}\right]\right)\\ 
& = & \sum c_q\cdot\varphi_{i+1}\vert_{\mathcal{D}_{i,q}}(\{q,w_q\})\nonumber\\ 
& = & \sum c_q\cdot \left(\sum_{q\lessdot y}\varphi_{i+1}^{y,q}(\{q,w_q\})\right)\nonumber\\ 
& = & \sum c_q\cdot \left(\sum_{q\lessdot y} \iota\circ\delta_{i-1}^{y,q}(\{q,w_q\})\right)\nonumber\\ 
& = & \sum c_q\cdot \left(\sum_{q\lessdot y} \left[d_{i-1}(\{q,w_q\})\right]\right)\nonumber\\
& = & \sum c_q\cdot \left(\sum_{q\lessdot y} [w_q]+[q,d_{i-2}(w_q)]\right)\nonumber\\
& = & \sum c_q\cdot \left(\sum_{q\lessdot y} [w_q]+[q,0]\right)\nonumber\\
& = & \sum c_q\cdot  [w_q]=v, \nonumber
\end{eqnarray}
where $[q,d_{i-2}(w_q)]=[q,0]$ since we assumed that 
$[w_q]$ is a cycle in $\Delta_q$. Thus, $v\in\image(\varphi_{i+1})$. 

If the class $\ds [u]=\left[\sum c_q\cdot \{q, w_q\}\right]$ 
does not generate $\Ho_{i-1}(\Delta^{L_M}_y,\Bbbk)$, then 
it certainly cannot generate $\Ho_{i-1}(\Delta^{B_M}_y,\Bbbk)$. 
There are now two possibilities. 

If $[u]$ generates $\Ho_{j}(\Delta^{L_M}_y,\Bbbk)$ in some 
dimension $j\ne i-1$, then we have found a $y$ which 
contributes in dimensions $j$ and $i-1$, which contradicts 
rigidity condition (R1). 

Suppose that $\ds [u]=\left[\sum c_q\cdot \{q, w_q\}\right]$ is 
zero in $\Ho_{i-1}(\Delta^{L_M}_y,\Bbbk)$. 
Thus, writing $\ds u=\sum c_q\cdot \{q, w_q\}$ we have 
$\ds d(u)=d\left(\sum c_q\cdot \{q, w_q\}\right)$ is the boundary of 
a simplicial chain in $\Delta^{B_M}_y\subset \Delta^{L_M}_y$. 
Calculating directly, we see that 
$$
	d(u) 
	= d\left(\sum c_q\cdot \{q, w_q\}\right)
	= \sum c_q\cdot \{w_q\}-\sum c_q\{q,d(w_q)\}
$$
must bound. We have assumed $[d(u)]=0$, 
so that passing to homology classes, 
$$
0=[d(u)]=\left[d\left(\sum c_q\cdot \{q, w_q\}\right)\right]
	= \left[\sum c_q\cdot \{w_q\}\right]-\left[\sum c_q\{q,d(w_q)\}\right]. 
$$
However $w_q$ is a homology cycle for every $q$, so $d(w_q)=0$ and therefore, 
$$
0 = [d(u)] = \left[\sum c_q\cdot \{w_q\}\right]=\sum c_q\cdot \left[w_q\right]. 
$$
Thus, we have shown that our kernel element $\ds v=\sum c_q\cdot \left[w_q\right]$ 
must be equal to zero in this subcase and therefore $v\in\image(\varphi_{i+1})$

\emph{Case 2:} 
Suppose $\Ho_{i-1}(\Delta^{L_M}_y,\Bbbk)=0$, so that $y\notin B_M$. 

If the class $\ds [u]=\left[\sum c_q\cdot \{q, w_q\}\right]$ 
generates $\Ho_{i-1}(\Delta^{L_M}_x,\Bbbk)$ for some $x$, then 
$x<y\in L_M$ and $x\in B_M$. If such an $x$ did not exist but $[u]$ generated homology 
then $y\in B_M$, contradicting our assumption that $y\notin B_M$. 

In fact, such an $x\in B_M$ which has $[u]$ as a generator for 
$\Ho_{i-1}(\Delta^{L_M}_x,\Bbbk)$ must be unique. 
To see this, suppose that $x$ and $x'$ are two elements of $B_M$ 
whose open interval homology is generated by $[u]$. 
Then $q<x$ and $q<x'$ for every $q$ in the sum. 
There are now two possibilities. First, if $x<x'$ (or $x>x'$), then 
we would have found two comparable elements in $B_M$ having nonzero 
Betti numbers in the same homological degree. This is a contradiction to 
rigidity condition (R2). 

On the other hand, if $x$ and $x'$ are incomparable, 
then within $B_M$, there exist at least two elements $a$ and $b$ such that 
 $a<x$ and $a\not<x'$, while $b<x'$ and $b\not<x$. 
(If these elements did not exist then either $x=x'$ or they are comparable.)
Since $M$ is rigid and $[u]$ has been assumed to be a nontrivial 
homology class, then any elements $a$ and $b$ which are present in the 
order complexes $\Delta_x^{B_M}$ and $\Delta_{x'}^{B_M}$ need not 
be part of any chain which generates the respective homologies. 
Indeed, were $a$ or $b$ part of such a chain, then one possibility 
is that such a chain generates homology separately from $[u]$. 
However, this contradicts rigidity condition (R1), since we assume 
that $[u]$ already generates a one-dimensional space without containing 
chains ending in $a$ or $b$. 
Alternately, if this new chain was homologous to $[u]$, 
and generated the homology of one of the order complexes 
$\Delta_x^{B_M}$ or $\Delta_{x'}^{B_M}$, but not the other, 
then we have a contradiction to our assumption that $[u]$ 
generates homology for both order complexes. 

Furthermore, $x\wedge x'\notin B_M$ when $x$ and $x'$ are incomparable 
elements of $B_M$ which are both comparable to all the $q$'s. 
Indeed, were this meet to exist in $B_M$, then a class generating 
the homology of the order complexes $\Delta_x^{B_M}$ and $\Delta_{x'}^{B_M}$ 
would not be carried by chains containing the $q$'s. 
In the lcm-lattice $L_M$, however, the meet $x\wedge x'$ must exists. 
Using the comparability relation in $L_M$, then certainly $x\wedge x'>q$ 
for every $q$ which indexes a summand in the original definition of the class $[u]$. 
However, since $x\wedge x'\notin B_M$ then by definition, 
$\Ho_{j}(\Delta^{L_M}_{x\wedge x'},\Bbbk)=0$ for every $j$. Hence, the homology of 
$\Delta_x^{L_M}$ and $\Delta_{x'}^{L_M}$ cannot be carried solely by the subcomplex 
$\Delta_{\cl{x\wedge x'}}$. However, any oriented chain which only contains 
poset chains ending in $q$'s is carried by this subcomplex. 
Thus, were $x$ and $x'$ distinct, then $[u]$ 
could not generate homology in the asserted dimension. 

Having established that when $[u]$ is a nontrivial homology cycle, it 
generates the homology of $\Delta^{B_M}_x$ for exactly one $x\in B_M$, 
we now apply the map $\varphi_{i+1}$. This calculation is similar to 
the one detailed in (\ref{E:PhiActingOnU}), where the covering element 
in this case is $x$. The associated equation implies that $v\in\image(\varphi_{i+1})$. 

Our final possibility is that the class $[u]$ does not generate $\Ho_{i-1}(\Delta^{B_M}_x,\Bbbk)$ 
for any $x\in B_M$. Together with the assumption that the $[w_q]$ generate 
homology in dimension $i-2$, rigidity guarantees that the class $[u]$ cannot generate 
homology in any dimension $j\ne i-1$. Hence, either $[u]=0$, or $[u]$ 
generates $i-1$ dimensional homology of $\Delta(B_M)$ since the elements $q$ would be 
maximal in $B_M$. The former case implies that $v=0$, which was argued 
in Case 1.  
We claim that the latter case is impossible due to the fact that 
for any poset $P$ with a maximum element, the Betti poset $B(P\smsm\{\hat{0}\})$ is acyclic.

Towards verification of this claim, write $\hat{1}$ for the maximum element of $P$. 
If $\hat{1}\in B(P)$, then the order complex 
$\Delta(B(P\smsm\{\hat{0}\}))$ is a cone with apex $\hat{1}$, and is acyclic. 

On the other hand, if $\hat{1}\notin B(P)$, then 
we proceed by induction on the number of non-contributing elements of $P$. 
If $\hat{1}$ is the unique non-contributing element of $P$, then $B(P\smsm\{\hat{0}\})$ 
must be acyclic. Indeed, if we assume otherwise, the equality 
$B(P\smsm\{\hat{0}\})=(\hat{0},\hat{1})\subset P$ would imply that the order complex 
$\Delta(\hat{0},\hat{1})=\Delta(B(P\smsm\{\hat{0}\}))$ had nontrivial homology, 
contradicting our assumption that $\hat{1}$ is a non-contributor. 

Suppose that $k>1$ and that for any poset $P'$ with $k$ non-contributors, the Betti poset  
$B(P'\smsm\{\hat{0}\})$ is acyclic. Let $P$ be a poset with $k+1$ non-contributors. 
For any non-contributor $x\ne\hat{1}$, Corollary \ref{SameBettis} guarantees 
an equality of Betti posets $B(P\smsm\{\hat{0}\}))=B(P\smsm\{\hat{0},x\})$. By the inductive hypothesis, 
$P\smsm\{x\}$ has $k$ non-contributors, so that $B(P\smsm\{\hat{0},x\})$ is acyclic. 
Hence, $B(P\smsm\{\hat{0}\}))$ is acyclic as claimed. 
Since an lcm-lattice $L_M$ has a maximum element, then $B(L_M\smsm\{\hat{0}\})$ is acyclic. 

This completes the proof of exactness. 
\end{proof}

\end{document}